\documentclass[makeidx]{amsart} 
\usepackage{amsmath}
\usepackage{amssymb}
\usepackage{color}
\usepackage{epsfig}
\usepackage{hyperref}
\newtheorem{Proposition}{Proposition}
  \newtheorem{Remark}{Remark}
  
  \newtheorem{Lemma}[Proposition]{Lemma}
  
  \newtheorem{Theorem}{Theorem}
 
 \newtheorem{Definition}[Proposition]{Definition}
 \newtheorem{Note}[Remark]{Note}

\def\L|{\left\|}
\def\R|{\right\|}
\def\epsilon{\varepsilon}
\def\blackslug{\hbox{\hskip 1pt \vrule width 4pt height 8pt depth 1.5pt
\hskip 1pt}}
\def\qed{\quad\blackslug\lower 8.5pt\null\par}

 \def\RR{\mathbb{R}}

\makeindex

\author{O. Costin$^1$, T. Kim$^{1}$ and S. Tanveer$^1$} \address{$^1$ Mathematics Department\\The Ohio State University\\Columbus, OH 43210 } 

\title{A quasi-solution approach to nonlinear problems--the case of Blasius Similarity solution}

\begin{document}
$ $ \vskip -0.2cm
\begin{abstract}
  Using the simple case of Blasius similarity solution,
  we illustrate a recently developed general method\cite{Costinetal},\cite{Dubrovin} 
  that reduces a strongly nonlinear problem into a weakly nonlinear analysis. The
  basic idea is to find a quasi-solution $F_0$ that satisfies the nonlinear problem and
  boundary conditions to within small errors. Then, by decomposing the true solution 
  $F=F_0+E$, a weakly nonlinear analysis of $E$, using contraction mapping theorem in
  a suitable space of functions provides the existence of solution as well as bounds on
  the error $E$. The quasi-solution construction relies on a combination of exponential
  asymptotics and standard orthogonal polynomial representations in finite domain.
  \end{abstract}

\vskip -2cm
\maketitle

\today

\section{Introduction and main results}

Nonlinear mathematical problems abound in vortex dynamics as they do in all
areas of the sciences. For the most part, available mathematical tools are limited
to numerical computations. While providing for valuable insights, computations do not usually address the
associated existence and uniqueness questions. In problems involving
infinite domain for which numerical computations are usually truncated to finite subdomains,
these questions are more than just of theoretical interest. For instance, determination of hetero-clinic
and homoclinic orbits of dynamical system play an important role in Langrangian chaos (see for instance
the review paper \cite{Aref}).
Yet, numerical computations of such two point boundary value problems 
cannot by themselves resolve the question whether or not such orbits exist in
the first place.  
Also, in many problems, such as in hydrodynamic 
stability, a clear understanding of the associated spectral problem is 
facilitated greatly by analytical representation of steady state solutions.       
This explains at least in part why 
analytical expressions for solutions to nonlinear problems remain an  
important area of research. 
Closed
form solutions however exist only for a small sub-class of problems
(essentially for integrable models). On the other-hand, if a
problem involves some small parameter $\epsilon$ (or a large
parameter) and the limiting problem is exactly solvable, then there
exist quite general asymptotic methods to obtain convenient
expansions for the perturbed problem.

Indeed, consider for instance the question of finding the solution to
$\mathcal{N} [ u ,\epsilon]=0$, where $\mathcal{N}$ is a (possibly nonlinear) differential operator 
in some space of functions satisfying boundary/initial
conditions and that $u_0$ is the solution at $\epsilon=0$. Existence and uniqueness of a solution
$u$ as well bounds on the error $E=u-u_0$ 
may be found as follows. We write
$$ L\,E=- \delta-\mathcal{N}_1(E)$$ 
where $L=\frac{\partial\mathcal{N}}{\partial u}|_{u=u_0}$ is the Frechet derivative of $\mathcal{N}$, 
$\delta = \mathcal{N} [u_0]$ is the residual and 
$\mathcal{N}_1(E)=\mathcal{N}(u_0+ E)-L\,E  =O(E^2)$. Assuming 
$L$ to be invertible in a suitable space of functions subject to appropriate initial/boundary
conditions, and using the fixed point or contractive mapping theorems in an adapted norm, 
the small nonlinearity $\mathcal{N}_1(E)$ can be controlled.

Recently, a relatively general strategy has been 
employed \cite{Costinetal}-\cite{Dubrovin} in problems without explicit small or large parameters. The approach 
uses exponential asymptotic methods and classical orthogonal polynomial techniques  
to find a function $u_0$, we call it a quasi-solution, 
which is a very accurate global approximation of the sought solution $u$, 
in the sense that 
$\delta=\mathcal{N}(u_0)$ 
is small in a suitable norm and the boundary conditions are satisfied to within small errors. 
Once this is accomplished, a 
perturbative approach similar with the one above  applies 
with the role of $\epsilon$  played by the norm of $\delta $ and one obtains  
an actual solution $u$ by controlling the equation satisfied by $E=u-u_0$.
The method has been generalized to
integro-differential equations arising in
steady 2-D deep water waves\cite{Water}; therefore, it might be
expected that this should be generalizable to
vortex patches as well. Indeed, quasi-solution approach is more general and
can be generalized to PDEs as well, though the details in higher dimensions
are computationally challenging.
The only crucial conceptual barrier is the ability to
determine suitably appropriate bounds on $\mathcal{L}^{-1}$.

Here, we explain 
the stategy for the relatively simple but  
well-known Blasius similarity solution \cite{Blasius} arising in boundary
layer fluid-flow past a flat plate.
Since the audience is mostly
non-mathematicians, we limit ourselves to presenting the theorems 
and explaining the implications while omitting technical proofs given elsewhere \cite{BlasiusCT}.
We also elucidiate the construction of quasi-solution and  
give some indication on how error estimates are obtained.
In the last section, we present new results when 
boundary layer similarity solution is required to satisfy a more general
boundary condition than the usual no-slip condition. The point is to briefly
explain how
parameters can be incorporated in a quasi-solution formulation.
The detailed proofs will appear elsewhere \cite{Kim}.

The classic Blasius similarity solution to boundary layer equations past a semi-infinite plate 
satisfies
\begin{equation}
\label{1}
f^{\prime \prime \prime} (x) + f (x) f^{\prime \prime} (x) = 0  ~~ \, {\rm for} ~x \in (0, \infty) 
\end{equation}
with no-slip boundary conditions:
\begin{equation}
\label{2}
f(0) = 0 \ , f^\prime (0) = 0 \ ,   \lim_{x \rightarrow +\infty} f^\prime (x) = 1 
\end{equation}
A generalization of (\ref{2}) is also of interest
(see for eg. \cite{Hussaini}, \cite{Brighi})
and involves modification of the no-slip boundary conditions:
\begin{equation}
\label{2.1.0}
f(0) = {\tilde \alpha} \ , f^\prime (0) = {\tilde \gamma} \ ,   \lim_{x \rightarrow +\infty} f^\prime (x) = 1 
\end{equation}
The Blasius similarity solution have garnered much attention since Blasius \cite{Blasius}
derived it\footnote{The equation
in the original Blasius' paper has a coefficient $\frac{1}{2}$ for $f f^{\prime}$;
however the
change of variable
$x \rightarrow x/\sqrt{2}$, $f \rightarrow f/\sqrt{2}$ transforms (\ref{1}) into Blasius' original equation. 
Thus, $f^{\prime \prime} (0) = 0.469600 \pm 0.000022$ 
transforms to $f^{\prime \prime} (0) = 0.3320574 \pm 0.000016$ 
in the original variables.} as an exact solution to Prandtl boundary layer equations.   
Existence and uniqueness were first proved 
by Weyl in \cite{Weyl}. Issues of existence and
uniqueness for this and related equations
have been considered as well by many 
authors (see for eg. \cite{Hussaini}, \cite{Brighi}, the latter being a review paper).
Hodograph transformations 
\cite{Callegari} allow a 
convergent power series representation in the entire domain,
but the convergence is slow at the edge of the domain and the
representation is not quite convenient in finding an approximation to $f$ directly.
Empirically, there has been quite a bit of interest in obtaining simple
expressions for Blasius and related similarity solutions. Liao \cite{Liao}
for instance
introduced a formal method for
an emprically accurate approximation; the theoretical
basis for this procedure and its limitations remain however unclear.
We are unaware of any rigorous error control for this or any
other efficient approximation in terms of simple functions.
Also, we are unaware of any systematic procedure that allows for
analytical representation of solution to any desired accuracy.

In \cite{Weyl}, using a transformation introduced
by Topfer \cite{Topfer}, 
it is also proved that $f$ in \eqref{1}, \eqref{2} can be expressed as
\begin{equation}
\label{2.1}
f (x) = a^{-1/2} F \left ( a^{-1/2} x \right ),
\end{equation}
where
$F$ satisfies the initial value problem
\begin{equation}
\label{3}
F^{\prime \prime \prime} (x) + F (x) F^{\prime \prime} (x) = 0  ~~ \, {\rm for} ~x \in (0, \infty) 
\end{equation}
with initial conditions
\begin{equation}
\label{4}
F(0) = 0 \ , F^\prime (0) = 0 \ ,   F^{\prime \prime} (0) = 1, 
\end{equation}
In \eqref{2.1}, $\lim_{x \rightarrow \infty} F^\prime (x)=a\in\RR^+$ (cf. \cite{Weyl,Topfer}).
More general boundary conditions (\ref{2.1.0}) on $f$ translate into
following initial conditions on $F$: 
\begin{equation}
\label{4.0}
F(0) = \alpha \ , F^\prime (0) = \gamma \ , F^{\prime \prime} (0) = 1  \ ,
\end{equation}
where $\alpha= a^{1/2} {\tilde \alpha}$ and $\gamma = a {\tilde \gamma}$.
Note that the solution $f$ to the original problem is obtained
through the transformation (\ref{2.1}); 
the appropriately non-dimensionalized wall stress is given by
\begin{equation}
\label{4.1}
f^{\prime \prime} (0) = a^{-3/2} 
\end{equation}
It is to be emphasized that this transformation, though convenient, is by no means necessary to
construct a quasi-solution, since a quasi-solution only needs to satisfy initial/boundary conditions
approximately. This will be clearer in the error analysis where it will be seen that
nonhomogeneous initial/boundary conditions
are allowable as long as they are small.
For this reason, the methodology outlined here can be extended to more
general two point boundary value
problems.

\section{Main Results for $(\alpha, \gamma) = (0,0)$}

Let
\begin{equation}
\label{6}
P(y) = \sum_{j=0}^{12} \frac{2}{5 (j+2)(j+3)(j+4)} p_j y^{j} \, 
\end{equation}
where $[p_0,...,p_{12}]$ are given by
{\small \begin{multline}
\Big[
-\frac{510}{10445149}, 
-\frac{18523}{5934}, -\frac{42998}{441819}, \frac{113448}{81151},-\frac{65173}{22093},\frac{390101}{6016},-\frac{2326169}{9858}, \\ \frac{4134879}{7249}, -\frac{1928001}{1960},
\frac{20880183}{19117} , -\frac{1572554}{2161} ,  \frac{1546782}{5833} , -\frac{1315241}{32239}\Big]
\end{multline}}
Define
\begin{equation}
\label{10.2}
t(x)=\frac{a}{2}(x+b/a)^2,
\ I_0 (t) = 1 - \sqrt{\pi t} e^{t} {\rm erfc} (\sqrt{t})  \ , ~~ 
J_0 (t) = 1 - \sqrt{2 \pi t} e^{2t} {\rm erfc} (\sqrt{2t}) \ ,
\end{equation}
where ${\rm erfc}$ denotes the complementary error function and let
\begin{equation}
\label{10.1}
q_0 (t) = 
2 c \sqrt{t} e^{-t} I_0  + 
c^2 e^{-2t} \left ( 2 J_0 - I_0 - I_0^2 \right ) ,
\end{equation}
The theorem below provides an accurate representation of solution $F$ to \eqref{3}, \eqref{4}.
\begin{Theorem}
\label{Thm1}Let $F_0$ be defined by
\begin{equation}
\label{6.0}
F_0 (x) =
\left\{ \begin{array}{cc}
\frac{x^2}{2} + x^4 P \left (\frac{2}{5} x \right ) \  \text{for $x\in [0,\frac52]$} \\
   a x + b + \sqrt{\frac{a}{2t (x)}} q_0 ( t(x)) \  \text{for $x > \frac52$}
\end{array}\right.
\end{equation}
Then, there is a unique triple $(a,b,c)$ close to $(a_0,b_0,c_0)= \left ( \frac{3221}{1946}, 
-\frac{2763}{1765}, \frac{377}{1613} \right )$ in the sense that $(a,b,c)\in\mathcal{S}$ where
\begin{equation}
\label{8}
\mathcal{S} = \left \{ (a, b, c)\in\RR^3: \sqrt{(a-a_0)^2+\frac{1}{4} (b-b_0)^2+
\frac{1}{4} (c-c_0)^2 } \le \rho_0 := 5 \times 10^{-5} \right \}
\end{equation}
with the property that $F_0$ is a representation of the actual solution $F$ 
to the initial value problem (\ref{3})-(\ref{4}) within small errors. More precisely,
\begin{equation}
\label{5}
F(x) = F_0 (x) + E (x)  \ ,
\end{equation}
where the error term $E$ satisfies
\begin{equation}
\label{errBoundsE}
\| E^{\prime \prime} \|_{\infty} \le 3.5 \times 10^{-6} \ ,  
\| {E}^{\prime} \|_{\infty} \le 4.5 \times 10^{-6}  \ ,  
\| {E} \|_{\infty} \le 4 \times 10^{-6}   \ \text{on $[0,\frac52]$}
\end{equation}
and for $x\ge \frac52$
\begin{multline}\label{eq15}
\Big |  {E}  \Big | 
\le 1.69 \times 10^{-5}  t^{-2} e^{-3 t}   \ ,    
\Big | \frac{d}{dx}  E  \Big | 
\le 9.20 \times 10^{-5} t^{-3/2} e^{-3 t}   \\
\Big | \frac{d^2}{dx^2}  E  \Big | 
\le 5.02 \times 10^{-4} t^{-1} e^{-3 t}   
\end{multline}
\end{Theorem}

\begin{Remark}{\rm 
Certainly, $F$ is smooth since it is is an actual 
solution of \eqref{3},\eqref{4}, which exists on $[0, \infty )$ 
and is unique, see \cite{Weyl}. 
However, the particular choice $(a, b, c) \in \mathcal{S}$ in Theorem \ref{Thm1} needed in order for 
$F=F_0+{E}$ to solve (\ref{3})-(\ref{4}) does not ensure continuity of
the approximate solution $F_0$ at $x=\frac{5}{2}$. Nonetheless, 
if  $F_0,F_0',F_0''$ are needed to be continuous, this can be achieved 
by a slightly different choice of $(a,b,c)\in \mathcal{S}$  (see Remark 
\ref{remCont}), namely
\begin{equation}
\label{eqabcCont}
(a, b, c) =  
\left (1.6551904561499...,-1.565439826457..,0.233728727537...
\right ). 
\end{equation}
}

\rm Note also that \eqref{eq15} implies 
not only small absolute errors 
(that, in the far field hold even for the approximation of $F^{\prime \prime}$ 
by zero) but also
very small {\em relative} errors on $\left [ \frac{5}{2}, \infty \right )$. 
\end{Remark}
\begin{Definition}{\rm 
\label{Def1}
Let $a_{l} = a_0 -\rho_0$, $a_r = a_0+\rho_0$, $b_l = b_0 - 2 \rho_0$,
$b_r = b_0 + 2 \rho_0$, $c_l = c_0 - 2 \rho_0$ and $c_r = c_0 + 2 \rho_0$. We see that $(a,b,c) \in \mathcal{S}$, implies that $a \in [a_l, a_r]$, $b \in [b_l, b_r]$,
$c \in (c_l, c_r)$.
We define $t_m = \frac{a}{2} \left ( \frac{5}{2}
+ \frac{b}{a} \right )^2 $ and note that $x \in [\frac{5}{2}, \infty )$ corresponds to
$t \in [t_m, \infty )$ and  $t_m \in 
\left ( \frac{a_l}{2} 
\left [ \frac{5}{2} + \frac{b_l}{a_l} \right ]^2 , 
\frac{a_r}{2} 
\left [ \frac{5}{2} + \frac{b_r}{a_r} \right ]^2 \right )  
=: \left (t_{m,l}, t_{m,r} \right ) = \left ( 1.998859\cdots, 1.999438 \cdots \right )$. 
}\end{Definition}

\begin{Remark}{\rm 
The error bounds proved for ${E}$ in Theorem \ref{Thm1}
are likely a 10 fold over-estimate. Comparison with the 
numerically calculated $F$ suggests that
${|}F-F_0 {|}$, ${|}F^{\prime}-F_0^\prime {|}$ and $
{|}F^{\prime \prime} - F_0^{\prime \prime} {|}$ in $\left [0, \frac{5}{2} \right ]$
are at most $2 \times 10^{-7}, 2 \times 10^{-7}$ and $5 \times 10^{-7}$ respectively. Using the nonrigorous bounds on ${E}$ and
its derivatives reduces the 
$\rho_0$ in the definition of $\mathcal{S}$ 
from $5 \times 10^{-5}$ to
$1.4 \times 10^{-5}$. 
It is thus likely that
$(a, b, c) \approx (a_0, b_0, c_0)$ with
five (rather than the proven four) digits accuracy. Further, there is no theoretical
limitation in the accuracy in this approach. Higher accuracy will require a higher
order or piecewise polynomial expressions in $\left [0, \frac{5}{2} \right ]$ and using
a higher order truncation of the series (\ref{14.3}) for $q_0$, as explained in the ensuing.}
\end{Remark}

The proof of Theorem \ref{Thm1} rests on the following three propositions; we will discuss
the idea behind the proofs of these propositions in later sections.
\begin{Proposition}
\label{Prop2}
The error term 
$E(x) = F(x) - F_0 (x)$  
verifies the equation
\begin{equation}
\label{eqE}
\mathcal{L} [E] := 
E^{\prime \prime\prime} - F_0 E^{\prime \prime} - F_0^{\prime \prime} E = - F_0^{\prime \prime \prime} - F_0 F_0^{\prime \prime} 
- E E^{\prime \prime}
\end{equation}
\begin{equation}
\label{eqEBC}
E(0) = 0 = E^{\prime} (0) = E^{\prime \prime} (0) 
\end{equation}
and satisfies the bounds 
(\ref{errBoundsE}) on $I = \left [ 0, \frac{5}{2} \right ]$
\end{Proposition}
\begin{Proposition}
\label{Prop3}
For given $(a, b, c)$ with $a > 0$, $ |c| < \frac{1}{4}$, in the domain 
$x \ge - \frac{b}{a} + \sqrt{\frac{2T}{a}}$, for $T \ge 1.99$, which
corresponds to 
$t \ge T \ge 1.99$, there exists unique solution to (\ref{3}) in the form 
\begin{equation}
\label{0.0}
F(x) = a x + b + \sqrt{\frac{a}{2 t(x)}} q (t(x))
\end{equation}
that satisfies the condition $\lim_{t \rightarrow \infty} \frac{q (t)}{\sqrt{t}} \rightarrow 0$.
Furthermore, 
\begin{equation}
q(t) = q_0 (t) + \mathcal{E} (t) 
\end{equation}
where $\mathcal{E}$ is small and satisfies the following error bounds:
\begin{equation}
\label{eqn:23}
\Big | \mathcal{E} (t) \Big | 
\le 1.6667 \times 10^{-4}  \frac{e^{-3 t}}{9 t^{3/2}}     
\end{equation}
\begin{equation}
\label{eqn:24}
\Big | \mathcal{E}^{\prime} (t) - 
\frac{1}{2t} \mathcal{E} (t) \Big |
\le 1.6667 \times 10^{-4}  \frac{e^{-3 t}}{3 t^{3/2}}     
\end{equation}
\begin{equation}
\label{eqn:25}
\Big | \sqrt{t} \mathcal{E}^{\prime \prime} (t) - 
\frac{1}{\sqrt{t}} \mathcal{E}^\prime (t) + \frac{1}{2 t^{3/2}} \mathcal{E} (t) \Big | 
\le 1.6667 \times 10^{-4}  t^{-1} e^{-3 t}    
\end{equation}
\end{Proposition}
\begin{Proposition}
\label{Prop4}
There exists a unique triple $(a, b, c) \in \mathcal{S}$ so that the functions 
in the previous two propositions: 
$F_0 (x) + E(x)$ for $x \le \frac{5}{2}$ 
and $ax+b + \sqrt{\frac{a}{2t(x)}} q (t(x)) $ for $x \ge \frac{5}{2}$ 
and their first two derivatives
agree at $x=\frac{5}{2}$.
\end{Proposition}

\noindent{\bf The proof} of Theorem \ref{Thm1} follows from Propositions 
\ref{Prop2}-\ref{Prop4} in the following manner:
Proposition \ref{Prop2} implies that $F (x) =F_0 (x) + E(x)$ satisfies 
\eqref{3}-\eqref{4} for $x \in I$; we note that $F_0 (0)=0 = F_0^\prime (0)$ 
and $F_0^{\prime\prime} (0)=1$. 

Proposition \ref{Prop3} implies $ F(x) = a x + b + \sqrt{\frac{a}{2 t(x)}} 
\left [ q_0 (t(x) ) + \mathcal{E} (t (x)) \right ]$
satisfies \eqref{3} in a range of $x$ that includes $[\frac{5}{2}, \infty)$
when $(a, b, c) \in \mathcal{S}$. Further, Proposition \ref{Prop4} ensures
that this is the same solution of the ODE \eqref{3} as the one in Proposition
\ref{Prop2}.
Identifying $F_0 (x) $ and $E(x)$ in 
Theorem \ref{Thm1} in this range of $x$ with 
$ a x + b + \sqrt{\frac{a}{2 t(x)}} q_0 (t(x))$ and  
$\sqrt{\frac{a}{2 t(x)}} \mathcal{E} (t(x))$, respectively,
and relating $x$-derivatives to $t$ derivatives,
the error bounds for $E$, $E^\prime$ and $E^{\prime \prime}$ follow from the ones
given for $\mathcal{E}$ in Proposition \ref{Prop3} for
$(a,b,c) \in \mathcal{S}$. We discuss these propositions in later sections. 

\section{Solution in the interval $ I = [0, \frac{5}{2}]$ for $(\alpha, \gamma) = (0,0)$
and proof of Proposition \ref{Prop2}}\label{S1}

The quasi-solution $F_0$ in the compact set $I$ is obtained simply by projecting
a numerical
solution 
on the subspace spanned by the first few Chebyshev polynomials 
$\left \{ T_j \left (\frac{4}{5} x - 1 \right ) \right \}_{j=0}^{N}$. 
To 
avoid estimating derivatives of an 
approximation, which are not well-controlled,  we 
project instead the approximate 
third derivative $F^{\prime\prime\prime}=-F^{\prime\prime}F$ on 
the interval
$I=[0, \frac{5}{2}]$. The rigorous control of the errors of the {\em integrals} 
of $F'''$ is a much simpler task.
For a given polynomial degree, a 
Chebyshev polynomial approximation of a function is known to be, 
typically, close to the 
most accurate polynomial approximation, in the sense of $L^\infty$.
A power series is less efficient since it is constrained
by complex plane behavior.

We seek to control the error term $E$ in (\ref{5}) by first estimating the remainder
\begin{equation}
\label{12.0}
R(x) = F_0^{\prime \prime\prime} (x) + F_0 (x) F_0^{\prime \prime} (x),
\end{equation} 
which will be shown to be small ($ \le 0.673 \times 10^{-6}$). 
Then, we invert the principal part
of the linear part of the equation for the error term $E$ by using initial
conditions to obtain a nonlinear integral equation. 
The smallness of $R$ and careful bounds on the resolvent $\mathcal{L}^{-1}$ help prove
Proposition \ref{Prop2}.

\subsection{Estimating size of remainder $R (x)$ for $ x\in I $}\label{R(x)}
Since $P$ is a polynomial of degree twelve,
$R(x)$ is a polynomial of degree 30. 
We estimate $R$ in $I$ in the following manner.
We break
up the interval into subintervals 
$\left \{ [x_{j-1}, x_{j}] \right \}_{j=1}^{14}$
with $x_0=0$ and $x_{14} = \frac{5}{2}$, while 
$\left \{x_j \right \}_{j=1}^{13}$ is given by
$$ \left \{ 0.0625, 0.125, 0.25, 0.375, 0.50, 0.75, 1.0, x_c, 1.5, 1.75, 2.0, 2.25, 2.40 \right \}$$
where $x_c = 1.322040$.\footnote{As described elsewhere\cite{Blasius}, it is convenient to choose one
of the subdivision points $x_c$ to be approximately, to the number of digits quoted,
the value of $x$ where 
$F_0^{\prime \prime} (x) - 2 F_0^\prime (x) +1$ changes sign.}
The intervals were chosen based 
on how rapidly the polynomial $R(x)$ varies locally.

We re-expand $R(x)$ as polynomial in the scaled variable   
$\tau$, where $x=\frac{1}{2} \left (x_j
+ x_{j-1} \right ) + \frac{1}{2} \left ( x_j-x_{j-1} \right ) \tau$.  
and write
$$ R(x) = P_3^{(j)}(\tau) + \sum_{k=4}^{30} a_k^{(j)} \tau^k
$$
and determine the maximum $M_j$ and minimum $m_j$ 
of the third degree polynomial $P_{3}^{(j)} (t)$
for $\tau \in [-1,1]$ (using simple calculus). 
We bound the contribution of the remaining terms:
$$ E_{R}^{(j)} \equiv \sum_{k=4}^{30} \Big | a_k^{(j)} \Big | $$  
It follows that in the $j$-subinterval we have
$$ m_j - E_R^{(j)} \le R(x) 
\le M_j + E_R^{(j)} $$
The maximum and minimum over any union of subintervals
is found simply taking $\min$ and $\max$ of $m_j - E_{R}^{(j)}$ and $M_j + E_{R}^{(j)}$ 
over the the indices $j$ for 
subintervals involved.   
This elementary though tedious calculation\footnote{The maximum and minimum found
through analysis described here is found to be consistent with a 
numerical plot of the graph of $R(x)$, as must be the case. 
The calculations can be conveniently done with a computer algebra program, as they only involve operations with rational numbers.} yields
\begin{multline}
\label{12.1}
-3.22 \times 10^{-7} \le R (x) \le 2.505 
\times 10^{-7} ~{\rm for} ~x \in [0, x_c ] \\
 4.6 \times 10^{-8}  \le  R (x) \le 4.06 \times 10^{-7}  
~{\rm for} ~x \in [x_c, 2.0] \\
 2.78 
\times 10^{-7}  
\le  R(x)  \le 6.73 \times 10^{-7}  ~{\rm for} ~x \in [2.0, 2.5]  
\end{multline}  
We note that the remainder is at most
$6.73 \times 10^{-7}$ in absolute value in the interval $I$.
In the same way, we find bounds for 
for the polynomials $F_0 (x)$, $F_0^\prime$ and $F_0^{\prime \prime} (x)$.
For $x \in \left [ 0 , \frac{1}{8} \right ]$,
\begin{multline}
\label{12.1.0}
-5 \times 10^{-10} \le F_0 (x) \le 0.008
\ ,  -8 \times 10^{-12} \le F_0^{\prime} (x) \le  
0.13   \ , \\  
0.99 \le F_0^{\prime \prime} (x)  \le 1 + 2 \times 10^{-9} 
\end{multline}
while for $x \in \left [ \frac{1}{8}, \frac{5}{2} \right ]$,
\begin{multline}
\label{12.1.1}
0.03 \le F_0 (x) \le 2.59 \ ,     0.12  \le F_0^\prime (x)  \le 1.7 \ ,
0.09 \le  F_0^{\prime \prime} (x) \le  1
\end{multline}

A less-unwieldly strategy for residual error estimation is to find $A_j$ so that
\begin{equation}
\label{15}
R (x)= \sum_{j=0}^{30} A_j T_j \left (\frac{4 x}{5} -1 \right )      
\end{equation}
Since for $y \in [-1,1]$, $T_j (y) = \cos \left ( j \cos^{-1} y \right )$ is less than 1 in absolute value
\begin{equation}
\label{15.1}
\| R \|_{\infty, I} \le \sum_{j=0}^{30} |A_j | \le 9.74 \times 10^{-7}
\end{equation}
Projecting $R$ instead to Chebyshev polynomials
in each of the sub-intervals $[0, x_c]$, $[x_c, 2]$ and $[2, 2.5]$ gives somewhat better bounds.
In both cases, the bounds are not as sharp as ones estimated through local Taylor series expansion.
Nonetheless, this method is simpler and more easily adapted to multi-variables.

\subsection{Error estimate on a sub-interval $[x_l, x_r ]\subset \mathcal{I}$}

Consider the decomposition
\begin{equation}
\label{12.1.0.0}
F(x) = F_0 (x) + E(x) 
\end{equation}
We seek to find error estimates for $E(x)$ and its first two derivatives for
$x \in I$.
On $[x_l, x_r] \subset I$, where $E(x_l)$, $E^\prime (x_l)$ and $E^{\prime \prime} (x_l)$
are considered known, $E$ satisfies:
\begin{equation}
\label{12.2}
\mathcal{L} [ E ] :=E^{\prime\prime \prime} (x) + F_0 (x) E^{\prime \prime} (x) + F_0^{\prime \prime} (x) 
E(x) = 
-E (x) E^{\prime \prime} (x) - R(x) 
\end{equation}
Using a variation of parameter approach, where $\left \{ \Phi_j \right \}_{j=1}^{3}$ are
fundamental solutions to $L \Phi = 0$, we may invert the operator $\mathcal{L}$ by using
the boundary condition at $x=0$ to obtain 
an integral equation in the form:
\begin{equation}
\label{12.3}
E^{\prime \prime} (x) = \sum_{j=1}^3 E^{(j-1)} (x_l) \Phi_j^{\prime \prime} (x)   
-\mathcal{G} \left [ R \right ] (x) +
\mathcal{G} \left [ E E^{\prime \prime} \right ] (x) 
=: \mathcal{N} \left [ E^{\prime \prime} \right ] (x)     
\end{equation}
and where  $E$ is given in terms of $E^{\prime \prime}$: 
\begin{equation}
\label{12.8.0}
E (x) =  E(x_l) + (x-x_l) E^{\prime} (x_l) + \int_{x_l}^x (x-t) E^{\prime \prime} (t) dt 
\end{equation}
Note that \eqref{12.8.0} allows control of $\| . \|_\infty$ (sup)-norm of
$E$ and $E^{\prime}$ in terms of $E^{\prime \prime}$. Error estimates follow by
showing that integral equation
(\ref{12.2}) written abstractly $E^{\prime \prime} = \mathcal{N} [E^{\prime \prime} ]$ 
has a unique solution in the space of continuous
functions in a small ball in the $\|. \|_\infty$ norm by showing that $\mathcal{N}$ is
contractive (see \cite{BlasiusCT} for details).
This is possible
without explicit knowledge of $\left \{ \Phi_j \right \}_{j=1}^3$ or the resolvent
operator $\mathcal{G}$, provided that the bounds are not too large.
In the following sub-section we detail how energy methods can be used for that purpose.

\subsection{Green's function estimate}
\label{subsec:Green}

Consider now the problem of solving the linear generally inhomogeneous equation
\begin{equation}
\label{12.0.0}
\mathcal{L} [ \phi ] :=\phi^{\prime\prime \prime} (x) + F_0 (x) \phi^{\prime \prime} (x) + F_0^{\prime \prime} (x)  \phi (x) = r(x)
\end{equation}
over a typical subinterval $[x_l, x_r] \subset I$,
with initial conditions $\phi (x_l)$, $\phi^\prime (x_l)$ and $\phi^{\prime \prime} (x_l)$ known.
The solution of this inhomogeneous equation is 
given by the standard variation of parameter formula:
\begin{equation}
\label{12.0.1}
\phi (x) = \sum_{j=1}^3 \phi^{(j-1)} (x_l) \Phi_j (x)  + \sum_{j=1}^3 \Phi_j (x) \int_{x_l}^x \Psi_j (t) r(t) dt 
\end{equation}
where $\left \{ \Phi_j \right \}_{j=1}^3 $ form a 
fundamental set of solutions to $\mathcal{L} \phi=0$ 
and $\left \{\Psi_j (x) \right \}_{j=1}^3$ are elements 
of the inverse of the fundamental matrix
constructed from the $\Phi_j$ and their derivatives. The 
precise expressions are unimportant in the ensuing: 
we only need their smoothness in  $x$.
It also follows from the  properties of
$\Phi_j$ and $\Psi_j$ \footnote{In particular, 
$\sum_{j=1}^3 \Phi_j (x) \Psi_j (x) =0$, 
$\sum_{j=1}^3 \Phi_j^\prime (x) \Psi_j (x)=0$}
that
\begin{equation}
\label{12.3.1}
\phi^{\prime \prime} (x) = 
\sum_{j=1}^3 \phi^{(j-1)} (x_l) \Phi_j^{\prime \prime} (x) 
+ \sum_{j=1}^3 \Phi_j^{\prime \prime} (x) \int_{x_l}^x \Psi_j (t) r(t) dt \ , 
\end{equation}
It is useful to write (\ref{12.3.1}) in the following abstract form
\begin{equation}
\label{12.4}
\phi^{\prime \prime} (x) = \sum_{j=1}^3 \phi^{(j-1)} (x_l) \Phi_j^{\prime \prime} (x) + \mathcal{G} \left [ r \right ] (x)      
\end{equation}
where from general properties of fundamental
matrix and its inverse for the linear ODEs with smooth 
(in this case polynomial) coefficients 
$\mathcal{G}$ is a bounded linear operator 
on $C([x_l, x_r])$; 
denote its norm by $M$,
\begin{equation}
  \label{eq:defM}
  M=\|\mathcal{G}\|
\end{equation}
 Then, on the interval $[x_l, x_r]$ we have,
\begin{equation}
\label{12.5}
\| \phi^{\prime \prime} \|_\infty \le  \sum_{j=1}^\infty M_j \Big | \phi^{(j-1)} (x_l) \Big | 
+ M \| r \|_{\infty};\ M_j
= \sup_{x \in [x_l,x_r]} 
\Big | \Phi_j^{\prime \prime} (x) \Big | 
\end{equation}
We will outline how estimates of
$M_j$ for $j=1..3$ and
$M$ may be obtained indirectly, 
using ``energy'' bounds. Because of 
linearity of the problem, for the purposes
of determining these bounds, it is useful to separately 
consider the cases (i)--(iii), when
$r =0$, $\phi^{(k-1)} (x_l) =0$ for $1 \le k \ne j \le 3$, 
$\phi^{(j-1)} (x_l) = 1$  respectively, and, finally, (iv) when
$\phi^{(k-1)} (x_l) =0$ for $k =1,..3$ and $r(t) \ne 0$. 
For all cases (i)-(iv), 
we return to the ODE
\begin{equation}
\label{12.13}
\phi^{\prime \prime \prime} + F_0 \phi^{\prime \prime} + F_0^{\prime \prime} \phi = r 
\end{equation}
Multiplying  by $2 \phi^{\prime \prime}$ and integration gives
\begin{equation}
\label{12.14}
\left ( \phi^{\prime \prime} (x) \right )^2 = \left ( \phi^{\prime \prime} (x_l) \right )^2 
- \int_{x_l}^x \left \{ 2 F_0 (y) \left ( \phi^{\prime \prime} (y) \right )^2  
+
2 F_0^{\prime \prime} (y) \phi^{\prime \prime} (y) \phi (y) - 
2 \phi^{\prime \prime} (y) r(y) \right \} dy,
\end{equation}     
Further, for given  $\phi (x_l)$ and $\phi^\prime (x_l)$,
$\phi(x)$ is determined from $\phi^{\prime \prime} (x)$ from
\begin{equation}
\label{12.15}
{\tilde \phi} (x) := 
\phi(x) - \phi (x_l) - (x-x_l) \phi^\prime (x_l) = \int_{x_l}^x (x-y) \phi^{\prime \prime} (y) dy
\end{equation}
Using (\ref{12.15}) in (\ref{12.14}), it follows that 
\begin{multline}
\label{12.14.1}
\left ( \phi^{\prime \prime} (x) \right )^2 = \left ( \phi^{\prime \prime} (x_l) \right )^2 
- \int_{x_l}^x 2 F_0^{\prime \prime} (y) 
\left [ \phi(x_l) + (y-x_l) \phi^{\prime} (x_l) \right ] \phi^{\prime \prime} (y) dy \\
- \int_{x_l}^x \left \{ 2 F_0 (y) \left ( \phi^{\prime \prime} (y) \right )^2  
+
2 F_0^{\prime \prime} (y) \phi^{\prime \prime} (y) {\tilde \phi} (y) - 
2 \phi^{\prime \prime} (y) r(y) \right \} dy,
\end{multline}     
Using Cauchy Schwartz inequalities, the relation between $\phi$ and $\phi^{\prime \prime}$ and
Gronwall's inequality, it is not difficult to prove by considering separately cases (i)-(iv) that
\begin{equation}
\label{12.14.9}
M_1 \le 
\left ( F_0^\prime (x_r) - F_0^{\prime} (x_l) \right )^{1/2}
\exp \left [ \frac{1}{2} \int_{x_l}^{x_r} Q_1 (y) dy \right ] \ ,
\end{equation}
\begin{equation}
\label{12.14.12.0}
M_2 \le 
\left ( \int_{x_l}^{x_r} (y-x_l)^2 F_0^{\prime \prime} (y) dy\right )^{1/2}
\exp \left [ \frac{1}{2} \int_{x_l}^{x_r} Q_1 (y) dy \right ] \ ,
\end{equation}
\begin{equation}
\label{12.18.0}
M_3 \le 
\exp \left [ \frac{1}{2} \int_{x_l}^{x_r} Q_2 (y) dy \right ] \ , 
\end{equation}
\begin{equation}
\label{12.14.18}
M   
\le \left ( x_r - x_l \right )^{1/2}
\exp \left [ \frac{1}{2} \int_{x_l}^{x_r} Q(y) dy \right ] \ ,  
\end{equation}
where
\begin{multline}
\label{12.14.6.0}
Q_1 (x) = F_0^{\prime \prime} (x) \left ( 2 + \frac{(x-x_l)^4}{4} \right ) - 2 F_0 (x)   
~~{\rm if}~~2 F_0^{\prime \prime} (x) - 2 F_0 (x) > 0   \\
\text{and} 
~ Q_1 (x) = \frac{(x-x_l)^4}{4} F_0^{\prime \prime} (x) ~~{\rm if}~~2 F_0^{\prime \prime} (x) - 2 F_0 (x) \le 0   
\end{multline}
\begin{multline}
\label{12.14.0}
Q_2 (x) = \left ( 1 + \frac{(x-x_l)^4}{4} \right ) F_0^{\prime \prime} (x) - 2 F_0 (x)  \ , 
~~{\rm if} ~~F_0^{\prime \prime} (x) - 2 F_0 (x) > 0  \\
\text{and} \ Q_2 (x) = \frac{(x-x_l)^4}{4} F_0^{\prime \prime} (x) \ , 
~~{\rm if} ~~F_0^{\prime \prime} (x) - 2 F_0 (x) \le 0  \ ,
\end{multline}
\begin{multline}
\label{12.14.10.0.0}
Q(x) = F_0^{\prime \prime} (x) - 2 F_0 (x) + 1 + \frac{(x-x_l)^4}{4} F_0^{\prime \prime} (x) ~~{\rm if} ~F_0^{\prime \prime} - 2 F_0 +1 \ge 0 
\\
Q(x) = \frac{(x-x_l)^4}{4} F_0^{\prime \prime} (x) ~~{\rm if} ~F_0^{\prime \prime} - 2 F_0 +1 < 0 
\end{multline}
It was possible\cite{BlasiusCT} to estimate $ M $, $M_1$, $M_2$, $M_3$ in three sub-intervals 
$[0, x_c]$, $[x_c, 2]$ and $[2, 2.5]$ and use them to get small error bounds
for $E$, $E^{\prime}$ and $E^{\prime \prime}$ to complete the proof of Proposition
\ref{Prop2}. 

\section{Solution in $t \ge T \ge 1.99 $ for $|c| < \frac{1}{4} $, $a > 0$
and proof of Proposition \ref{Prop3}}
\label{S2}

The construction of quasi-solution $F_0$ for $x \in \left [ \frac{5}{2}, \infty \right ]$ 
relies on precise large $x$ asymptotics, which as it turns out, gives a desirably accurate solution representation
in the entire interval. 
For the Blasius solution, it is known that any solution with $\lim_{x \rightarrow \infty} F^\prime (x)   
= a > 0$ must have the representation 
\begin{equation} 
F(x) = a x+ b + G(x) 
\end{equation}
where $G(x)$ is exponentially small in $x$ for large $x$. Indeed, through change of
variable $t=t(x)$ given by  
(\ref{10.2}) and $G = \sqrt{\frac{a}{2 t}} q(t)$, $q$ satisfies
\begin{equation}
\label{14.2}
{\frac {d^{3}}{d{t}^{3}}}q + \left ( 1 +  
\frac {q}{2t} \right )  {\frac {d^{2}}{d{t}^{2}}}q
+ \left( -\frac {1}{2t}+\frac{3}{4 t^2} -
\frac {q}{4 t^{2}} \right) \frac {dq}{dt} 
+ \left ( \frac{1}{2 t^2} -\frac{3}{4 t^3} \right ) q  
 +\frac{q^2}{4 t^3} = 0 
\end{equation}
and from a general theory\cite{Duke}\footnote{Though the non-degeneracy condition 
stated in \cite{Duke} does not hold, a small modification leads to the same result}
it may be deduced
that small solutions $q$ must have the convergent series representation
\begin{equation}  
\label{14.3}
q (t) = \sum_{n=1}^\infty \xi^n Q_n (t)  \ , {\rm where}~ \xi= \frac{c e^{-t}}{\sqrt{t}}  
\end{equation}
where the equations for $Q_n$ may be deduced by plugging in (\ref{14.3}) into (\ref{14.2}) and equating
different powers of $\xi$. With appropriate matching
at $\infty$, one obtains $Q_1 (t) = 2 t I_0 (t)$ and
$Q_2 (t) = - t I_0 - t I_0^2 + 2 t J_0$, where
\begin{equation}
\label{14.4}
I_0 = 1 - \sqrt{\pi t} e^{t} {\rm erfc} (\sqrt{t}) 
\end{equation} 
\begin{equation}
\label{14.5}
J_0 = 1 - \sqrt{2 \pi t} e^{2 t} {\rm erfc} (\sqrt{2 t}) 
\end{equation}
The two term truncation of (\ref{14.3}) proved adequate to determine an accurate quasi-solution 
in an $x$-domain that corresponds to $t \ge 1.99$ if $|c| \le \frac{1}{4}$ to within the quoted
accuracy. Note that the solution is only complete 
after determining $(a, b, c)$ through 
matching of $F$, $F^\prime$ and $F^{\prime\prime}$ at $x=\frac{5}{2}$.
Since $(a, b, c)$ only needs to be restricted to some 
small neighborhood of $(a_0, b_0, c_0)$ to accomplish matching (see Proposition \ref{Prop4}), 
the restriction 
$t \ge 1.99$ is seen to include $x \ge \frac{5}{2}$. Furthermore, the restriction $|c| \le \frac{1}{4}$
in Proposition \ref{Prop4} is appropriate for the quoted
error estimates in $x \ge \frac{5}{2}$ in Theorem \ref{Thm1}.

We decompose 
\begin{equation}
\label{14.3.1} 
q(t) = q_0 (t) + \mathcal{E} (t) \ ,
\end{equation}
where 
\begin{equation}
\label{14.3.0}
q_0 (t) = \frac{c e^{-t}}{\sqrt{t}} Q_1 (t) + \frac{c^2 e^{-2t}}{t} Q_2 (t) \ , 
\end{equation}
where 
\begin{equation}
\label{14.5b}
Q_1 (t)  = 2 t I_0  (t) \ , ~{\rm where} ~I_0 (t) := 1 - \sqrt{\pi t} e^{t} 
{\rm erfc} (\sqrt{t}) = \frac{1}{2} \int_0^\infty \frac{e^{-st}}{(1+s)^{3/2}} ds \ ,
\end{equation}
\begin{equation} 
\label{14.5.0}
Q_2 (t) =  
- t I_0 - t I_0^2 + 2 t J_0 \ , ~{\rm where} ~ J_0 (t) 
:= 1 - \sqrt{2 \pi t} e^{2 t} {\rm erfc} (\sqrt{2 t}) = \frac{1}{4} \int_0^\infty 
\frac{e^{-st}}{(1+s/2)^{3/2}} ds  \ .
\end{equation}
We obtain a nonlinear integral equation for $h$, which is related to $\mathcal{E}$ as follows:
\begin{equation}
\label{14.3.0.0}
\mathcal{E} (t) = \sqrt{t}
\int_{\infty}^t \frac{ds}{\sqrt{s}} \int_{\infty}^s \frac{e^{-\tau}}{\sqrt{\tau}} h(\tau) d\tau
\end{equation}
A contraction mapping argument in a small ball is possible by exploiting the smallness of the residual  
$R=R (t)$ given by  
\begin{equation}
\label{14.5R}
R = 
{\frac {d^{3}}{d{t}^{3}}}q_0 + \left ( 1 +  
\frac {q_0}{2t} \right )  {\frac {d^{2}}{d{t}^{2}}}q_0
+ \left( -\frac {1}{2t}+\frac{3}{4 t^2} -
\frac {q_0}{4 t^{2}} \right) \frac {dq_0}{dt} 
+ \left ( \frac{1}{2 t^2} -\frac{3}{4 t^3} \right ) q_0  
 +\frac{q_0^2}{4 t^3} 
\end{equation}
This leads to proof of Proposition \ref{Prop3} (see \cite{BlasiusCT} for details).

\section{Matching for $(\alpha, \gamma) = (0,0)$ and proof of Proposition \ref{Prop4}}
\label{S3}

In order for the two representations \eqref{12.1.0.0} and $F(x) = a x+b + 
\sqrt{\frac{a}{2t(x)}} \left ( q_0 (t(x)) + \mathcal{E} (t(x)) \right ) $ to 
to coincide at $x = \frac{5}{2}$ 
we match $F$ and its two derivatives; from (\ref{14.3.0}), (\ref{14.3.1}) and (\ref{14.3.0.0}) we get
\begin{equation}
\label{16}
a = F^\prime (\tfrac{5}{2}) 
- a \left ( q_0^\prime (t_m; c) - \frac{q_0 (t_m; c)}{2t_m} \right ) -  
a \int_{\infty}^{t_m} \frac{e^{-\tau}}{\sqrt{\tau}} h (\tau; c) d\tau 
=:N_1 (a, b, c)
\end{equation}
\begin{equation}
\label{15b}
b = F(\tfrac{5}{2}) - \frac{5}{2} N_1 - \sqrt{\frac{a}{2 t_m}} q_0 (t_m; c)   
- \sqrt{\frac{a}{2}} 
\int_{\infty}^{t_m} \tau^{-1/2} \int_{\infty}^\tau s^{-1/2} e^{-s} h(s; c) ds 
:=N_2 (a, b, c)
\end{equation}
\begin{equation}
\label{17}
c =  
\frac{1}{\sqrt{2} a^{3/2}}
\left [ V (t_m; c) + \frac{1}{c} h (t_m; c) \right ]^{-1}  
e^{t_m} F^{\prime \prime} (\tfrac{5}{2}) =: N_3 (a, b, c) 
\end{equation}      
\begin{Definition}
\label{defn:5}
We define ${\bf A} = \left ( a , \frac{b}{2}, \frac{c}{2} \right )$, ${\bf A}_0 = \left ( a_0, \frac{b_0}{2}, \frac{c_0}{2} \right )$
and ${\bf N} ({\bf A} ) = \left ( N_1, \frac{1}{2} N_2, 
\frac{1}{2} N_3 \right )$. Define also
$$\mathcal{S}_A := \left \{ 
\| {\bf A} - {\bf A}_0 \|_2 \le \rho_0:=5 \times 10^{-5} \right \}$$
where $\|. \|_2$ is the Euclidean norm and let
\begin{equation}
\label{17.3.0}
{\bf J} = {\begin{pmatrix} \partial_a N_1 & 2 \partial_b N_1  & 2 \partial_c N_1 \cr
         \frac{1}{2} \partial_a N_2 &  \partial_b N_2 & \partial_c N_2 \cr
         \frac{1}{2} \partial_a N_3 &  \partial_b N_3 & \partial_c N_3 
\end{pmatrix}}
\end{equation}
\end{Definition}
\begin{Note}{\rm }
  We see that ${\bf A} \in \mathcal{S}_A $ implies $(a, b, c) \in \mathcal{S}$.  
The system of equations (\ref{16})-(\ref{17}) is written as
\begin{equation}
\label{15.0.0}
{\bf A} = \mathbf{N} [{\bf A} ]
\end{equation}
We define
${\bf J} = \frac{\partial{\bf N}}{\partial {\bf A}}$ to be
the Jacobian and $\|{\bf J} \|_2$ denotes the
$l^2 $ (Euclidean) norm of the matrix. We note that
\begin{multline}
\label{17.3.0.0}
\| J \|^2_2 = 
\left ( \partial_a N_1 \right )^2 + 4 \left (\partial_b N_1\right )^2  + 4 \left ( \partial_c N_1 \right )^2
+\frac{1}{4} \left ( \partial_a N_2 \right )^2 
+\left (\partial_b N_2\right )^2  + \left ( \partial_c N_2 \right )^2
\\
+\frac{1}{4} \left ( \partial_a N_3 \right )^2 
+\left (\partial_b N_3\right )^2  + \left ( \partial_c N_3 \right )^2
\end{multline}
\end{Note}

\begin{Lemma}
\label{lemMatch}{\rm 
The inequalities
\begin{equation} \label{162}
\| {\bf A}_0 - {\bf N} [{\bf A}_0] \|_2 \le (1-\alpha) \rho_0
\end{equation}
\begin{equation}\label{163}
\sup_{{\bf A} \in \mathcal{S}_A} \| {\bf J} \|_2 \le \alpha < 1
\end{equation}
for some $\alpha \in (0, 1)$ imply that
${\bf A} = \mathbf{N} [ {\bf A} ]$
has a unique solution for ${\bf A} \in \mathcal{S}_A$.}
\end{Lemma}
\begin{proof}
The mean-value theorem implies 
\begin{multline}
\| {\bf N} [{\bf A} ] - {\bf A}_0 \|_2 
\le \| {\bf N} [{\bf A}_0] - {\bf A}_0 \|_2 + 
\| {\bf N} [{\bf A}] - {\bf N} [{\bf A}_0 ] \|_2  
\le \rho_0 (1-\alpha) + \| {\bf J} \|_2 \rho_0   
\le \rho_0 
\end{multline}
and also, if ${\bf A}_1, {\bf A}_2 \in \mathcal{S}_A$:
$$
\| {\bf N} [{\bf A}_1 ] 
- {\bf N} [{\bf A}_2 ] \|_2 
\le \|{\bf J} \|_2 \| {\bf A}_1 - {\bf A}_2 \|_2  
\le \alpha \|{\bf A}_1 - {\bf A}_2 \|_2 
$$
Thus, \eqref{162} and \eqref{163} imply that $\mathbf{N}: \mathcal{S}_A 
\rightarrow \mathcal{S}_A$ and that it is contractive there; the result follows from the contractive mapping theorem. 
\end{proof}
\subsection{Proof  of Proposition \ref{Prop4} }
Proposition \ref{Prop4} follows from Lemma \ref{lemMatch}
once we we verify \eqref{162} and \eqref{163} hold. 
This has been 
shown \cite{BlasiusCT} for
$\alpha \le 0.764$ and that $ \| {\bf A}_0 - \mathbf{N} [{\bf A}_0] \|_2 
\le 1.16 \times 10^{-5} \le (1-\alpha) \rho_0$ thereby completing
the proof of Proposition
\ref{Prop4}.

\begin{Remark} 
\label{remCont}
\rm{
Note that the proof of Proposition \ref{Prop4}  only requires smallness of the norms of $h$ and $E$  (we recall that  $F=F_0+E$) and on no further details about them. If in some application  $F_0$ needs to be made $C^2$, then this can be ensured by iterating $\mathbf{N}$ with $h=E=0$;  the
first thirteen digits obtained in this way are given in (\ref{eqabcCont}).

}
\end{Remark}

\section{Generalization for $(\alpha, \gamma) \ne (0,0)$}

Here we consider for simplicity 
the special case $\gamma=0$, $\alpha \in \left [-\frac{3}{50}, \frac{3}{50} \right ]$.
Through piecewise polynomial representatons, other intervals in $\alpha$ can similarly
be incorporated; it is to be noted that non-existence of globally acceptable solution for some
ranges of $(\alpha, \gamma)$ is manifest
in the present
approach by lack of matching at $x=\frac{5}{2}$.

Let 
\begin{equation}
P\left(y;\beta\right)=\sum_{i=0}^{13}\sum_{j=0}^{5}\frac{p_{i,j}}{\left(i+1\right)\left(i+2\right)\left(i+3\right)}\beta^{j}y^{i}
\end{equation}
$ $where $p_{i,j}$ is the $\left(i+1,j\right)$-entry of the following
matrix 

\begin{equation}
\left[\begin{array}{cccccc}
{\frac{29589}{493148}} & -{\frac{9845}{82042}} & -{\frac{274}{40132715}} & {\frac{241}{11270972}} & -{\frac{422}{16143111}} & {\frac{308}{28130517}}\\
\noalign{\medskip}{\frac{15185}{1706376}} & -{\frac{17096}{473735}} & {\frac{36599}{968864}} & -{\frac{19441}{3418968}} & {\frac{6287}{892276}} & -{\frac{10649}{3570017}}\\
\noalign{\medskip}-{\frac{203116}{65155}} & -{\frac{3042}{970153}} & -{\frac{15440}{235863}} & {\frac{21239}{89058}} & -{\frac{114887}{372923}} & {\frac{5024}{37953}}\\
\noalign{\medskip}-{\frac{72804}{75433}} & {\frac{239497}{147253}} & {\frac{213995}{192583}} & -{\frac{110079}{28121}} & {\frac{1322305}{259224}} & -{\frac{80021}{35684}}\\
\noalign{\medskip}{\frac{106800}{43663}} & -{\frac{112122}{86717}} & -{\frac{155285}{19732}} & {\frac{525204}{17519}} & -{\frac{2029749}{49136}} & {\frac{391166}{20741}}\\
\noalign{\medskip}-{\frac{387344}{32609}} & {\frac{77473}{4402}} & {\frac{304475}{15867}} & -{\frac{3049469}{26658}} & {\frac{445437}{2501}} & -{\frac{568723}{6514}}\\
\noalign{\medskip}{\frac{3084825}{27611}} & -{\frac{1006071}{9319}} & {\frac{171511}{4286}} & {\frac{3723623}{24721}} & -{\frac{1097313}{2915}} & {\frac{1207261}{5453}}\\
\noalign{\medskip}-{\frac{2254258}{5883}} & {\frac{3595213}{9561}} & -{\frac{1049674}{2379}} & {\frac{2081034}{4399}} & {\frac{1013365}{19943}} & -{\frac{1249672}{5459}}\\
\noalign{\medskip}{\frac{1915077}{2126}} & -{\frac{3165632}{3527}} & {\frac{5196992}{3543}} & -{\frac{3429722}{1327}} & {\frac{3839299}{2153}} & -{\frac{2755673}{9363}}\\
\noalign{\medskip}-{\frac{2860297}{1927}} & {\frac{3706169}{2627}} & -{\frac{5245388}{1929}} & {\frac{1764108}{317}} & -{\frac{6522639}{1366}} & {\frac{1111693}{833}}\\
\noalign{\medskip}{\frac{281944}{179}} & -{\frac{3174435}{2257}} & {\frac{5003871}{1621}} & -{\frac{7633149}{1117}} & {\frac{6098777}{958}} & -{\frac{9281007}{4606}}\\
\noalign{\medskip}-{\frac{2506157}{2481}} & {\frac{2704059}{3157}} & -{\frac{8285683}{3873}} & {\frac{6455381}{1295}} & -{\frac{4186545}{863}} & {\frac{3106817}{1912}}\\
\noalign{\medskip}{\frac{2072736}{5813}} & -{\frac{1425478}{4881}} & {\frac{3778762}{4529}} & -{\frac{980233}{486}} & {\frac{3100252}{1537}} & -{\frac{4063417}{5821}}\\
\noalign{\medskip}-{\frac{1051227}{19699}} & {\frac{745495}{17357}} & -{\frac{1839247}{13071}} & {\frac{1844827}{5276}} & -{\frac{2241089}{6290}} & {\frac{3813801}{30274}}
\end{array}\right].
\end{equation}
 
\begin{Theorem}
For any $\alpha\in\left[-\frac{3}{50},\frac{3}{50}\right]$, let $F_{0}$
be defined by 
\begin{equation}
F_{0}\left(x;\alpha\right)=\begin{cases}
\alpha+\frac{x^{2}}{2}+x^{3}P\left(\frac{2}{5}x;\frac{25}{3}\alpha+\frac{1}{2}\right)\quad & x\in\left[0,\frac{5}{2}\right]\\
ax+b+\sqrt{\frac{a}{2t\left(x\right)}}q_{0}\left(t\left(x\right)\right) & x\in\left(\frac{5}{2},\infty\right)
\end{cases}
\end{equation}
where $t\left(x\right)$ and $q_{0}\left(t\right)$ are as defined
in \eqref{10.2} and \eqref{10.1}. Let $a_{0}\left(\alpha\right)$, $b_{0}\left(\alpha\right)$,
and $c_{0}\left(\alpha\right)$ be defined by 
\begin{eqnarray}
a_{0}\left(\alpha\right) & = & {\frac{3221}{1946}}-{\frac{797}{603}}\,\alpha+{\frac{176}{289}}\,{\alpha}^{2}\\
b_{0}\left(\alpha\right) & = & -{\frac{2763}{1765}}+{\frac{761}{284}}\,\alpha-{\frac{194}{237}}\,{\alpha}^{2}\\
c_{0}\left(\alpha\right) & = & {\frac{377}{1613}}+{\frac{174}{1357}}\,\alpha+{\frac{937}{6822}}\,{\alpha}^{2}.
\end{eqnarray}
Then, for each $\alpha\in\left[-\frac{3}{50},\frac{3}{50}\right]$,
there exists a unique triple $\left(a,b,c\right)$ close to $\left(a_{0}\left(\alpha\right),b_{0}\left(\alpha\right),c_{0}\left(\alpha\right)\right)$
in the sense that $\left(a,b,c\right)\in\mathcal{S}_{\alpha}$ where
\begin{equation}
\mathcal{S}_{\alpha}=\left\{ \left(a,b,c\right)\in\mathbb{R}^{3}:\sqrt{\left(a-a_{0}\left(\alpha\right)\right)^{2}+\frac{1}{4}\left(b-b_{0}\left(\alpha\right)\right)^{2}+\frac{1}{4}\left(c-c_{0}\left(\alpha\right)\right)^{2}}\le\rho_{0}:=5\times10^{-4}\right\} 
\end{equation}
with the property that $F_{0}$ is a representation of the actual
solution $F$ to the initial value problem \eqref{3} and \eqref{4.0} with $\gamma=0$ and  $\alpha\in\left[-\frac{3}{50},\frac{3}{50}\right]$ within small errors. More precisely,
\begin{equation}
F\left(x;\alpha\right)=F_{0}\left(x;\alpha\right)+E\left(x;\alpha\right),
\end{equation}
where the error term $E$ satisfies 
\begin{equation}
\label{Error finite}
\Vert E^{\prime\prime}\Vert_{\infty}\le4.90\times10^{-6},\Vert E^{\prime}\Vert_{\infty}\le3.75\times10^{-6},\Vert E\Vert_{\infty}\le7.50\times10^{-6}\quad\text{on \ensuremath{\left[0,\frac{5}{2}\right]}, }
\end{equation}
where $\Vert E\Vert_{\infty}:=\sup\left\{ \left|E\left(x;\alpha\right)\right|:x\in\left[0,\frac{5}{2}\right],\alpha\in\left[-\frac{3}{50},\frac{3}{50}\right]\right\} $,
and for $x\ge\frac{5}{2}$, 
\begin{gather}
\left|E\left(x;\alpha\right)\right|\le1.76\times10^{-5}t\left(x\right)^{-2}e^{-3t\left(x\right)},\left|\frac{\partial}{\partial x}E\left(x;\alpha\right)\right|\le9.82\times10^{-5}t\left(x\right)^{-3/2}e^{-3t(x)},\nonumber \\
\left|\frac{\partial^{2}}{\partial x^{2}}E\left(x;\alpha\right)\right|\le5.50\times10^{-4}t\left(x\right)^{-1}e^{-3t\left(x\right)}.\label{Error far field}
\end{gather}
\end{Theorem}

\subsection{Results in the interval $\mathcal{I}=\left[0,\frac{5}{2}\right]$}

Let $\mathcal{I}:=\left[0,\frac{5}{2}\right]$ and $\mathcal{J}:=\left[-\frac{3}{50},\frac{3}{50}\right]$.
For any $\alpha\in\mathcal{J}$, the residual $R\left(x;\alpha\right)$
on the interval $x\in\mathcal{I}$ is given by 
\begin{eqnarray}
R\left(x;\alpha\right) & = & \frac{\partial^{3}}{\partial x^{3}}F_{0}\left(x;\alpha\right)+F_{0}\left(x;\alpha\right)\frac{\partial^{2}}{\partial x^{2}}F_{0}\left(x;\alpha\right).
\end{eqnarray}
Let $\left\{ x_{k}\right\} _{k=0}^{4}=\left\{ 0,1.25,1.4,2,2.5\right\} .$
Then $\mathcal{I}_{k}:=\left[x_{k-1},x_{k}\right]$, $k=1,..,4$,
provide a partition of $\mathcal{I}$.

\subsubsection{Sizes of $R$}

Let the sup-norms of $R$ over different subintervals be distinguished
as follows: 
\begin{gather}
\Vert R\Vert_{\infty,\mathcal{I}_{k}}:=\sup\left\{ \left|R\left(x;\alpha\right)\right|:x\in\mathcal{I}_{k},\alpha\in\mathcal{J}\right\} \,\text{for}\, k=1,...,4.
\end{gather}
The following estimates are obtained from the $l^{1}$-norm of the
coefficients of Chebyshev expansion of $R$ on appropriate intervals
of $x$ and $\alpha$ as described in Subsection \ref{R(x)}.

\begin{gather}
\Vert R\Vert_{\infty,\mathcal{I}_{1}}\le5.18\times10^{-7},\quad\Vert R\Vert_{\infty,\mathcal{I}_{2}}\le7.55\times10^{-7}\label{eq:Rbnd 1,2}\\
\Vert R\Vert_{\infty,\mathcal{I}_{3}}\le1.31\times10^{-6},\quad\Vert R\Vert_{\infty,\mathcal{I}_{4}}\le2.94\times10^{-6}.\label{eq:Rbnd 3,4}
\end{gather}
As a result, $\Vert R\Vert_{\infty,\mathcal{I}}:=\sup\left\{ \left|R\left(x;\alpha\right)\right|:x\in\mathcal{I},\alpha\in\mathcal{J}\right\} $
is smaller than $2.94\times10^{-6}$.

\subsubsection{Various supremema}
\def\I{\mathcal{I}}
\def\J{\mathcal{J}}
\def\eps{\varepsilon}

These figures are obtained by using 'energy' bounds as shown in Subsection
\ref{subsec:Green}, in particular, by using the inequalities \eqref{12.14.9}-\eqref{12.14.18}. The
modified construction of functions, $Q_{1},Q_{2},$ and $Q$ appearing
in \eqref{12.14.9}-\eqref{12.14.18} is found in \cite{Kim}.


\begin{table}
\caption{The bounds of various suprema on subregions $\I_k\times\J$.}
\label{tab:suprema}
\begin{center}
\begin{tabular}{|c||c|c|c|c|}
\hline                              
 & $M$ & $M_{1}$ & $M_{2}$ & $M_{3}$\tabularnewline
\hline
\hline
$\I_{1}\times \J$ & $3.1930$ & $3.0482$ & $2.1323$ & $1.5886$\tabularnewline
$\I_{2}\times \J$ & $0.3912$ & $0.3323$ & $0.0284$ & $1.0001$\tabularnewline
$\I_{3}\times \J$ & $0.7762$ & $0.5465$ & $0.1701$ & $1.0020$\tabularnewline
$\I_{4}\times \J$ & $0.7077$ & $0.3120$ & $0.0775$ & $1.0008$\tabularnewline
\hline
\end{tabular}
\par\end{center}
\end{table}

\subsubsection{Error estimates}

Using the contractive mapping principle \cite{BlasiusCT}, \cite{Kim} and the above
result, we obtain the following error bounds on the subintervals.
Compared with the errors against numerically calculated solution, these
estimates turned out to be 10 to 20-fold over-estimates. 


\begin{table}
\caption{The error estimates on subintervals}
\label{tab:error}
\begin{center}
\begin{tabular}{|c|c|c|c|c|c|}
\hline 
 & $B_{0}$ & $\eps$ & $\Vert E\Vert_{\infty,\mathcal{I}_{j}}$ & $\Vert E^{\prime}\Vert_{\infty,\mathcal{I}_{j}}$ & $\Vert E^{\prime\prime}\Vert_{\infty,\mathcal{I}_{j}}$\tabularnewline
\hline 
\hline 
$\mathcal{I}_{1}$ & $1.6538\times10^{-6}$ & $5\times10^{-6}$ & $1.6538\times10^{-6}$ & $2.0673\times10^{-6}$ & $1.2921\times10^{-6}$\tabularnewline
\hline 
$\mathcal{I}_{2}$ & $2.4371\times10^{-6}$ & $7\times10^{-7}$ & $2.4371\times10^{-6}$ & $3.6556\times10^{-7}$ & $1.6296\times10^{-6}$\tabularnewline
\hline 
$\mathcal{I}_{3}$ & $4.3873\times10^{-6}$ & $3\times10^{-6}$ & $4.3873\times10^{-6}$ & $2.6324\times10^{-6}$ & $2.6386\times10^{-6}$\tabularnewline
\hline 
$\mathcal{I}_{4}$ & $7.4947\times10^{-6}$ & $4\times10^{-6}$ & $7.4947\times10^{-6}$ & $3.7474\times10^{-6}$ & $4.8916\times10^{-6}$\tabularnewline
\hline 
\end{tabular}
\par\end{center}
\end{table}

This yields the error bounds \eqref{Error finite}.

\subsection{Results in the interval $t\ge T\ge1.96$}

Let $a_{l}\left(\alpha\right)=a_{0}\left(\alpha\right)-\rho_{0}$,
$a_{r}\left(\alpha\right)=a_{0}\left(\alpha\right)+\rho_{0}$, $b_{l}\left(\alpha\right)=b_{0}\left(\alpha\right)-2\rho_{0}$,
$b_{r}\left(\alpha\right)=b_{0}\left(\alpha\right)+2\rho_{0}$, $c_{l}\left(\alpha\right)=c_{0}\left(\alpha\right)-2\rho_{0}$,
$c_{r}\left(\alpha\right)=c_{0}\left(\alpha\right)+2\rho_{0}$. Define
$t_{m}=\frac{a}{2}\left(\frac{5}{2}+\frac{b}{a}\right)^{2}$ and note
that $x\in[\frac{5}{2},\infty)$ corresponds to $t\in[t_{m},\infty)$.
Note that, for each $\alpha\in\mathcal{J}$, $\left(a,b,c\right)\in\mathcal{S}_{\alpha}$
implies that $a\in\left[a_{l}\left(\alpha\right),a_{r}\left(\alpha\right)\right]$,
$b\in\left[b_{l}\left(\alpha\right),b_{r}\left(\alpha\right)\right]$,
and $c\in\left[c_{l}\left(\alpha\right),c_{r}\left(\alpha\right)\right]$.
Since $a_{0}$ and $b_{0}$ are quadratic functions in $\alpha$,
simple calculations show that $0<a_{l}\left(\alpha\right)$ and $b_{r}\left(\alpha\right)<0$
on the interval $\mathcal{J}$, and so 
\begin{equation}
\frac{a_{l}\left(\alpha\right)}{2}\left(\frac{5}{2}+\frac{b_{l}\left(\alpha\right)}{a_{l}\left(\alpha\right)}\right)^{2}<t_{m}<\frac{a_{r}\left(\alpha\right)}{2}\left(\frac{5}{2}+\frac{b_{r}\left(\alpha\right)}{a_{r}\left(\alpha\right)}\right)^{2},\quad\alpha\in\mathcal{J}.
\end{equation}
Thus it must be the case that 
\begin{gather}
t_{m}\in\left(\inf_{\alpha\in\mathcal{J}}\left\{ \frac{a_{l}\left(\alpha\right)}{2}\left(\frac{5}{2}+\frac{b_{l}\left(\alpha\right)}{a_{l}\left(\alpha\right)}\right)^{2}\right\} ,\sup_{\alpha\in\mathcal{J}}\left\{ \frac{a_{r}\left(\alpha\right)}{2}\left(\frac{5}{2}+\frac{b_{r}\left(\alpha\right)}{a_{r}\left(\alpha\right)}\right)^{2}\right\} \right)\nonumber \\
=\left(1.962257\cdots,2.043219\cdots\right).
\end{gather}
So, provided that $a>0$, the domain $t\ge T\ge1.96$ corresponds
to the domain $x\ge-\frac{b}{a}+\sqrt{\frac{2T}{a}}$, with $T\ge1.96$.
On this domain, the inequalities \eqref{eqn:23}-\eqref{eqn:25} are modified to have the constant
$1.6955\times10^{-4}$ in place of $1.667\times10^{-4}$. With these
modified inequalities, we obtain \eqref{Error far field}.
For the details of the proof and the computation, see \cite{Kim}. 
\[
\]

\subsection{Mathcing of the two solutions}

For each $\alpha\in\mathcal{J}$, define $\mathbf{A}_{0}\left(\alpha\right)=\left(a_{0}\left(\alpha\right),b_{0}\left(\alpha\right),c_{0}\left(\alpha\right)\right)$.
Let $\mathbf{A}$, $\mathbf{N}\left[\mathbf{A}\right]$,
and $\mathbf{J}$ be defined as in Definition \ref{defn:5}. In addition,
for each $\alpha\in\mathcal{J}$, define
\begin{equation}
\mathcal{S}_{\mathbf{A},\alpha}=\left\{ \mathbf{A}\in\mathbb{R}^{3}:\Vert\mathbf{A}-\mathbf{A}_{0}\left(\alpha\right)\Vert_{2}\le\rho_{0}:=5\times10^{-4}\right\} 
\end{equation}
where $\Vert\cdot\Vert_{2}$ is the Euclidean norm. \cite{Kim} shows that
for each $\alpha\in\mathcal{J}$ 
\begin{gather}
\label{Ndiff}\Vert\mathbf{A}_{0}\left(\alpha\right)-\mathbf{N}\left[\mathbf{A}_{0}\left(\alpha\right)\right]\Vert_{2}\le4.15\times10^{-5}\\
\label{J}\sup_{\mathbf{A}\in\mathcal{S}_{A}}\Vert\mathbf{J}\left(\alpha\right)\Vert_{2}\le0.839.
\end{gather}
\eqref{Ndiff} and \eqref{J} satisfy the conditions in Lemma \ref{lemMatch} for each $\alpha\in\mathcal{J}$
and it follows that the two solutions match at $x=\frac{5}{2}$ for
any value of $\alpha\in\mathcal{J}$. 

\section{Acknowledgments} The work of O.C. and S.T  was partially supported by the NSF
grant DMS 1108794.

\vfill \eject
\end{document}